\documentclass[10pt]{article}

\usepackage{amsmath,amssymb,amsthm}

\renewenvironment{proof}{\noindent {\bfseries Proof.}}{\hfill $\Box$}

\usepackage{a4wide}

\usepackage{cite}

\makeatletter
\def\@biblabel#1{#1.}
\makeatother

\newtheorem{theorem}{Theorem}[section]
\newtheorem{corollary}{Corollary}[section]
\newtheorem{lemma}{Lemma}[section]

\newtheorem{definition}{Definition}[section]
\newtheorem{remark}{Remark}[section]
\newtheorem{ex}{Example}[section]

\begin{document}

\title{Necessary Optimality Conditions for Higher-Order Infinite Horizon
Variational Problems on Time Scales}

\author{Nat\'{a}lia Martins$^{1}$
\and
Delfim F. M. Torres$^{2}$}

\date{(Communicated by Boris S. Mordukhovich)}


{\renewcommand{\thefootnote}{}

\footnotetext{{\bf This is a preprint of a paper whose final and definite 
form will appear in \emph{Journal of Optimization Theory and Applications}
(JOTA). Paper submitted 17-Nov-2011; revised 24-March-2012 and 10-April-2012;
accepted for publication 15-April-2012.}}


\footnotetext{This work was supported by FEDER funds through
COMPETE --- Operational Programme Factors of Competitiveness
(``Programa Operacional Factores de Competitividade'')
and by Portuguese funds through the
Center for Research and Development
in Mathematics and Applications (University of Aveiro)
and the Portuguese Foundation for Science and Technology
(``FCT --- Funda\c{c}\~{a}o para a Ci\^{e}ncia e a Tecnologia''),
within project PEst-C/MAT/UI4106/2011
with COMPETE number FCOMP-01-0124-FEDER-022690.
Torres was also supported by project PTDC/MAT/113470/2009.}}

\footnotetext[1]{Assistant Professor
in the Department of Mathematics, University of Aveiro, 3810-193 Aveiro, Portugal.
Senior Researcher in the Center for Research and Development in Mathematics and Applications,
Department of Mathematics, University of Aveiro, 3810-193 Aveiro, Portugal.
E-mail: natalia@ua.pt.}

\footnotetext[2]{Corresponding author.
Associate Professor with Habilitation in the Department of Mathematics,
University of Aveiro, 3810-193 Aveiro, Portugal.
Coordinator of the research group on Mathematical Theory of Systems and Control,
Center for Research and Development in Mathematics and Applications,
Department of Mathematics, University of Aveiro, 3810-193 Aveiro, Portugal.
E-mail: delfim@ua.pt.}


\maketitle


\begin{abstract}
We obtain Euler--Lagrange and transversality optimality conditions
for higher-order infinite horizon variational problems on a time scale.
The new necessary optimality conditions improve the classical
results both in the continuous and discrete settings:
our results seem new and interesting even in the particular cases
when the time scale is the set of real numbers or the set
of integers.

\bigskip


\bigskip

\end{abstract}


\section{Introduction}

We consider infinite horizon variational
problems on time scales, which consist in maximizing a delta integral
with a Lagrangian involving higher-order delta derivatives
on a given unbounded time scale.
Problems of the calculus of variations of such type
have many applications in economics both in discrete
(\textrm{i.e.}, when the time scale is the set of integers)
and continuous (\textrm{i.e.}, when the time scale is the set of real numbers)
time settings (see, \textrm{e.g.}, \cite{book:economics,livro:Sethi}).
Indeed, the dynamic processes of economics are usually described
with discrete or continuous models. The time scale approach
adopted here puts discrete and continuous models of economics together
and, most important, extends them to more realistic situations of
unequally spaced points in time (time-varying graininess).
Consider a typical situation of a consumer, that has to make decisions
concerning how much to consume and how much to spend with the goal
to maximize his lifetime utility subject to certain constraints.
Problems of the calculus of variations on time scales
provide a natural way to model such a consumer, that
has an income from different sources, at unequal
time intervals, and makes expenditures also at unequal time intervals \cite{Atici:2006}.
The reader interested on the usefulness of the calculus of variations on time scales in economics
is referred to \cite{Atici:2006,Ric:delfim:JVC,Atici:2008,Malinowska+Torres:comp}
and references therein.

Clearly, for infinite horizon variational problems, the delta integral
does not necessarily converge: it may diverge to plus or
minus infinity or it may oscillate.
In such situations, the extension of the standard definition of
optimality used in the time scale setting (see, \textrm{e.g.},
\cite{Bohner:2004,Rui+Barbara+Delfim})
to the unbounded time domain is not useful. Indeed,
if, for example, for every admissible function the value of the integral
functional is equal to plus infinity, then each admissible path could be called an
optimal path. To handle this and similar situations in a rigorous
way, several alternative definitions of optimality for problems with
unbounded time domain have been proposed in the literature (see,
\textrm{e.g.}, \cite{Brock,Gale,SS,Weiz}). In this paper, we follow
the notion of weakly optimal solution introduced by Brock
in the economic literature. In the case when
the variational functional converges for all admissible paths,
Brock's notion coincides with the standard definition of optimality.
Many results in infinite horizon optimal control
with this type of optimality can be found in the book \cite{MR2164615}.
For the method of discrete approximations, that allows to approximate continuous-time control problems
by those associated with discrete dynamics, we refer the reader to \cite{MR2191745}.

The goal of this paper is to provide
necessary optimality conditions
to higher-order infinite horizon variational problems on time scales.
Our main result is Theorem~\ref{main:result}.
It provides a nontrivial generalization of the recent results
of \cite{MMT-2010,naty:irlanda}. Moreover, Theorem~\ref{main:result}
improves the continuous results of Okomura et al. \cite{Nitta-et-all-2009}
when one chooses the time scale to be the set of real numbers,
while, in the particular case when the time scale is the set of integers,
it generalizes the discrete-time results of Cai and Nitta \cite{Nitta-et-all-2010}.

The paper is organized as follows. In Section~\ref{Preliminary results},
we present some preliminary results and basic definitions necessary in the sequel.
Main results are given in Section~\ref{sec:mr}:
in Section~\ref{Fundamental_Lemmas}, we prove some fundamental lemmas
of the calculus of variations for infinite horizon variational problems;
the Euler--Lagrange equation and the transversality conditions
for higher-order infinite horizon variational problems are proved
in Section~\ref{E-L_and_Transversality} and discussed in
Section~\ref{E-L_and_Transversality:disc}. We end the paper
with two illustrative examples (Section~\ref{sec:il:ex})
and a summary of the major results (Section~\ref{sec:conc}).


\section{Preliminaries}
\label{Preliminary results}

A time scale is an arbitrary, nonempty and closed subset
$\mathbb{T}$ of $\mathbb{R}$ (endowed with the topology of a
subspace of $\mathbb{R}$). In a time scale $\mathbb{T}$,
we consider the following two operators: the forward jump operator
$\sigma:\mathbb{T}\rightarrow\mathbb{T}$, defined by
$\sigma(t):=\inf{\{s\in\mathbb{T}:s>t}\}$ if $t\neq \sup
\mathbb{T}$ and $\sigma(\sup \mathbb{T}):=\sup \mathbb{T}$,
and the backward jump operator
$\rho:\mathbb{T}\rightarrow\mathbb{T}$, defined by
$\rho(t):=\sup{\{s\in\mathbb{T}:s<t}\}$ if $t\neq \inf \mathbb{T}$
and $\rho(\inf \mathbb{T}):=\inf \mathbb{T}$.
A point $t\in\mathbb{T}$ is called right-dense,
right-scattered, left-dense or left-scattered
if and only if $\sigma(t)=t$, $\sigma(t)>t$,
$\rho(t)=t$ or $\rho(t)<t$, respectively. We say that $t$ is
isolated if and only if $\rho(t)<t<\sigma(t)$, $t$ is dense if and only if
$\rho(t)=t=\sigma(t)$. The mapping $\mu:\mathbb{T}\rightarrow[0,+\infty[$ is
defined by $\mu(t):= \sigma(t)-t$ and is called the graininess function.

In order to introduce the definition of delta
derivative, we define a new set $\mathbb{T}^\kappa$.
If $\mathbb{T}$ has a left-scattered maximum $M$, then
$\mathbb{T}^\kappa=\mathbb{T}\setminus\{M\}$, otherwise,
$\mathbb{T}^\kappa= \mathbb{T}$.

\begin{definition}
We say that a function
$f:\mathbb{T}\rightarrow\mathbb{R}$ is delta differentiable
at $t\in\mathbb{T}^\kappa$ if and only if there is a number $f^{\Delta}(t)$
such that, for all $\varepsilon>0$, there exists a neighborhood $U$
of $t$ such that
$$
|f(\sigma(t))-f(s)-f^{\Delta}(t)(\sigma(t)-s)|
\leq\varepsilon|\sigma(t)-s| \mbox{ for all $s\in U$}.
$$
We call $f^{\Delta}(t)$ the delta derivative of $f$ at $t$.
Moreover, we say that $f$ is delta differentiable (or
$\Delta$-differentiable) on $\mathbb{T}$ provided
$f^{\Delta}(t)$ exists for all $t \in \mathbb{T}^\kappa$.
\end{definition}

\begin{remark}
If $\mathbb{T}=\mathbb{R}$, then
$f^{\Delta}=f^{\prime}$, where $f^{\prime}$ denotes the
usual derivative on $\mathbb{R}$. If $\mathbb{T}=\mathbb{Z}$, then
$f^{\Delta}=f(t+1)-f(t)$,  \textrm{i.e.}, $f^{\Delta}$ is the
usual forward difference. For any time scale $\mathbb{T}$,
if $f$ is a constant, then
$f^{\Delta}=0$; if $f(t)=k t$ for some constant $k$,
then $f^{\Delta}=k$.
\end{remark}

In order to simplify expressions, we denote the composition
$f\circ \sigma$ by $f^{\sigma}$.

\begin{theorem}{\rm \cite{Bohner-Peterson1}}
\label{propriedades derivada delta}
Let $\mathbb{T}$ be a time scale,
$f:\mathbb{T}\rightarrow\mathbb{R}$, and
$t\in\mathbb{T}^\kappa$. The following holds:
\begin{enumerate}
\item If $f$ is delta differentiable at $t$, then $f$ is
    continuous at $t$.
\item If $f$ is continuous at $t$ and $t$ is right-scattered,
    then $f$ is delta differentiable at $t$ and
$$f^{\Delta}(t)=\frac{f^{\sigma}(t)-f(t)}{\mu(t)}.$$

\item If $t$ is right-dense, then $f$ is delta differentiable
    at $t$ if and only if the limit
$$
\lim_{s\rightarrow t} \frac{f(t)-f(s)}{t-s}
$$
exists as a finite number. In this case,
$$
f^\Delta(t)=\lim_{s\rightarrow t} \frac{f(t)-f(s)}{t-s}.
$$
\item If $f$ is delta differentiable at $t$, then
$f^{\sigma}(t)=f(t)+\mu(t)f^\Delta(t)$.
\end{enumerate}
\end{theorem}

\begin{definition}
Let $f,F:\mathbb{T}\rightarrow\mathbb{R}$.
Function $F$ is called a delta antiderivative of
$f$ if and only if $F^{\Delta}(t)=f(t)$
for all $t \in \mathbb{T}^\kappa$.
In this case we define the delta integral of $f$ from $a$
to $b$ ($a,b \in \mathbb{T}$) by
$$
\int_{a}^{b}f(t)\Delta t:=F(b)-F(a).
$$
\end{definition}

\begin{definition}
A function $f:\mathbb{T}\rightarrow\mathbb{R}$
is rd-continuous if and only if it is continuous at the
right-dense points and its left-sided limits exist (finite) at all
left-dense points. The set of all rd-continuous functions
$f:\mathbb{T}\rightarrow\mathbb{R}$ is denoted by
$C_{\textrm{rd}}(\mathbb{T}, \mathbb{R})$.
\end{definition}

\begin{theorem}{\rm \cite{Bohner-Peterson1}}
Every rd-continuous function $f:\mathbb{T}\rightarrow\mathbb{R}$ has a delta
antiderivative. In particular, if $a \in \mathbb{T}$, then the
function $F$ defined by
$$
F(t) := \int_{a}^{t}f(\tau)\Delta\tau, \quad t \in \mathbb{T} \, ,
$$
is a delta antiderivative of $f$.
\end{theorem}

\begin{theorem}{\rm \cite{Bohner-Peterson1}}
\label{propriedades delta integral}
If $a,b,c \in \mathbb{T}$, $a
\le c \le b$, $\alpha \in \mathbb{R}$, and $f,g \in
C_{\textrm{rd}}(\mathbb{T}, \mathbb{R})$, then
\begin{enumerate}
\item $\displaystyle \int_{a}^{b}\left(f(t) + g(t) \right)
    \Delta t= \int_{a}^{b}f(t)\Delta t +
    \int_{a}^{b}g(t)\Delta t$;

\item $\displaystyle \int_{a}^{b} \alpha f(t)\Delta t =\alpha
    \int_{a}^{b}f(t)\Delta t$;

\item $\displaystyle \int_{a}^{b}  f(t)\Delta t = -
    \int_{b}^{a} f(t)\Delta t$;

\item $\displaystyle \int_{a}^{a}  f(t)\Delta t=0$;

\item $\displaystyle \int_{a}^{b}  f(t)\Delta t =
    \int_{a}^{c}  f(t)\Delta t + \int_{c}^{b} f(t)\Delta t$;

\item if $f(t)> 0$ for all $a < t \leq b$, then $
    \displaystyle \int_{a}^{b}  f(t)\Delta t > 0$;

\item If $f$ and $g$ are $\Delta$-differentiable, then

\begin{enumerate}

\item $\displaystyle \int_{a}^{b}f^\sigma(t)g^{\Delta}(t)\Delta t
=\left[(fg)(t)\right]_{t=a}^{t=b}-\int_{a}^{b}f^{\Delta}(t)g(t)\Delta t$;

\item $\displaystyle \int_{a}^{b}f(t)g^{\Delta}(t)\Delta t
=\left[(fg)(t)\right]_{t=a}^{t=b}-\int_{a}^{b}f^{\Delta}(t)g^\sigma(t)\Delta t$.
\end{enumerate}
\end{enumerate}
\end{theorem}

For more definitions, notations, and results concerning the theory of time scales, we refer
the reader to the books \cite{Bohner-Peterson1,Bohner-Peterson2}. In what follows, $\sigma$
denotes the forward jump operator and $\Delta$ is the delta derivative of a given time scale $\mathbb{T}$.
As usual, for $f:\mathbb{T}\rightarrow \mathbb{R}$ we define
$f^{\sigma^k}:=f\circ \sigma^k$, where
$\sigma^k:=\sigma \circ \sigma^{k-1}$, $k \in \mathbb{N}$, and $\sigma^0=id$.
We assume that $\mathbb{T}$ is a time scale
such that $\sup \mathbb{T} = +\infty$ and we
suppose that $a, T, T^\prime\in \mathbb{T}$ are such that $T^\prime \geq T > a$.
Let $r \in \mathbb{N}$ and $f^{\Delta^{0}}:=f$. The $r$th delta derivative
of $f:\mathbb{T}\rightarrow \mathbb{R}$ is the function
$f^{\Delta^r}:\mathbb{T}^{\kappa^r}\rightarrow \mathbb{R}$
defined by $f^{\Delta^r}:=\left(f^{\Delta^{r-1}}\right)^\Delta$,
provided $f^{\Delta^{r-1}}$ is delta differentiable.
By $\partial_i F$  we denote the partial derivative of
a function $F$ with respect to its $i$th argument.
All intervals are time scales intervals, that is, we simply write
$[a,b]$ and $[a,+\infty[$ to denote, respectively, the set
$[a,b]\cap \mathbb{T} $ and $[a,+\infty[ \, \cap \, \mathbb{T}$.
We consider the following
higher-order variational problem on $\mathbb{T}$:
\begin{equation}
\label{problem}
\begin{gathered}
\int_{a}^{+\infty} L\left(t,x^{\sigma^{r}}(t), x^{\sigma^{r-1}\Delta}(t), \ldots,
x^{\sigma\Delta^{r-1}}(t), x^{\Delta^{r}}(t)\right) \Delta t  \longrightarrow \max \\
x \in C^{2r}_{rd}\left([a,+\infty[,\mathbb{R}\right)\\
x^{\Delta^{i}}(a)=\alpha_{i}, \quad i = 0, \ldots, r-1,
\end{gathered}
\end{equation}
where $r \in \mathbb{N}$, $\alpha_0, \ldots, \alpha_{r-1}$ are fixed real numbers,
$(u_0,\ldots, u_{r})\rightarrow L(t,u_0,\ldots,u_{r})$ is a
$C^1(\mathbb{R}^{r+1}, \mathbb{R})$ function for any $t \in [a,+\infty[$,
and $\partial_{i+2} L \in C^{r}_{rd}([a,+\infty[,\mathbb{R})$ for all $i=0,\ldots,r$.

\begin{remark}
The results of this paper are trivially generalized for functions
$x:[a,+\infty[\rightarrow\mathbb{R}^n$ ($n \in \mathbb{N}$), but for simplicity of
presentation we restrict ourselves to the scalar case ($n=1$).
\end{remark}

\begin{definition}
We say that $x$ is an admissible function
for problem \eqref{problem} if and only if
$$
x \in C^{2r}_{rd}([a,+\infty[,\mathbb{R}) \text{ and }
x^{\Delta^{i}}(a)=\alpha_{i}, \, i = 0, \ldots, r-1.
$$
\end{definition}

As optimality criteria, we use the following
generalization of Brock's notion of optimality.

\begin{definition}
\label{def:weakMax}
Function $x_{\ast}$ is weakly maximal to problem
\eqref{problem} if and only if $x_{\ast}$ is admissible and
\begin{multline*}
\lim_{T\rightarrow+\infty}\inf_{T^\prime \geq
T}\int_{a}^{T^\prime}\left[L\left(t,x^{\sigma^{r}}(t), x^{\sigma^{r-1}\Delta}(t), \ldots,
x^{\sigma\Delta^{r-1}}(t), x^{\Delta^{r}}(t)\right)\right.\\
\left. - L\left(t,x_{\ast}^{\sigma^{r}}(t), x_{\ast}^{\sigma^{r-1}\Delta}(t), \ldots,
x_{\ast}^{\sigma\Delta^{r-1}}(t), x_{\ast}^{\Delta^{r}}(t)\right)\right]\Delta t \le 0
\end{multline*}
for all admissible function $x$.
\end{definition}

It is well known that, for certain time scales $\mathbb{T}$,
the forward jump operator $\sigma$ is not delta differentiable.
Furthermore, the chain rule, as we know it from the classical calculus, that is,
when $\mathbb{T}=\mathbb{R}$, is not valid in general.
However, if we suppose that the time scale $\mathbb{T}$
satisfies the condition
\begin{description}
\item[$(H)$] \quad \quad for each $t \in\mathbb{T}$, $(r-1) \left(\sigma(t) - a_1t - a_0\right) = 0$ for some
$a_1\in\mathbb{R}^+$ and $a_0\in\mathbb{R}$,
\end{description}
then we can deal with these two limitations as noted in Remark~\ref{rem:rest:H}
and Lemma~\ref{lemmaderivadacomposta}.

\begin{remark}
\label{rem:rest:H}
Condition $(H)$ is equivalent to $r = 1$ or $\sigma(t) = a_1t + a_0$
for some $a_1\in\mathbb{R}^+$ and $a_0\in\mathbb{R}$. Thus,
for the first order infinite horizon variational problem \cite{MMT-2010},
we impose no restriction on the time scale $\mathbb{T}$. For the higher-order problems
(\textrm{i.e.}, for $r \ge 2$) such restriction on the time scale is necessary.
Indeed, for $r > 1$ we are implicitly assuming in \eqref{problem}
that $\sigma$ be delta differentiable, which is not true for a general
time scale $\mathbb{T}$. Note that, for $r > 1$,
condition $(H)$ implies that $\sigma$ be delta
differentiable and $\sigma^{\Delta}(t)=a_1$,
$t \in \mathbb{T}$. Furthermore, note that
condition $(H)$ includes the following important cases:
the differential calculus ($\mathbb{T}=\mathbb{R}$, $a_1=1$, $a_0=0$);
the difference calculus ($\mathbb{T}=\mathbb{Z}$, $a_1=1$, $a_0=1$);
the $h$-calculus ($\mathbb{T}=h \mathbb{Z} := \{ h z: z \in
\mathbb{Z}\}$ for some $h>0$, $a_1=1$, $a_0=h$);
and the $q$-calculus ($\mathbb{T}= q^{\mathbb{N}_0} := \{ q^k: k \in
\mathbb{N}_0\}$ for some $q>1$, $a_1=q$, $a_0=0$).
\end{remark}

\begin{lemma}{\rm \cite{Rui+Delfim}}
\label{lemmaderivadacomposta}
Let $\mathbb{T}$ be a time scale satisfying condition $(H)$ and
$r>1$. If
$f:\mathbb{T}\rightarrow \mathbb{R}$ is two times delta
differentiable, then
$f^{\sigma\Delta}(t)=a_1 f^{\Delta\sigma}(t)$, $t\in\mathbb{T}$.
\end{lemma}

The next lemma will be very useful for the proof of our higher-order
fundamental lemmas of the calculus of variations on time scales
(more precisely, will be useful for Lemma~\ref{Fund.Lemma.1} and Lemma~\ref{Fund.Lemma.3}).
An analogous nabla version can be found in \cite{Martins+Torres-2009}.

\begin{lemma}
\label{lemma_funcoes_admissiveis_1}
Assume that the time scale $\mathbb{T}$ satisfies condition $(H)$
and $\eta \in C^{2r}_{rd}([a,+\infty[, \mathbb{R})$ is such that
$\eta^{\Delta ^{i}}(a)=0$ for all $i=0,\ldots, r$.
Then, $\eta^{\sigma \Delta^{i-1}}(a)=0$ for each $i=1,\ldots, r$.
\end{lemma}

\begin{proof}
If $a$ is right-dense, then the result is trivial
(just use Lemma~\ref{lemmaderivadacomposta} and the fact that $\sigma(a)=a$).
Suppose that $a$ be right-scattered and fix $i \in \{1,\ldots, r\}$. Since
$$
\eta^{\Delta^{i}}(a)
= \left(\eta^{\Delta^{i-1}}\right)^{\Delta}(a)
= \frac{
\left(\eta^{\Delta^{i-1}}\right)^\sigma(a)- \eta^{\Delta^{i-1}}(a)}{\sigma(a)-a},
$$
$\eta^{\Delta^{i}}(a)=0$, and $\eta^{\Delta^{i-1}}(a)=0$, then
$\left(\eta^{\Delta^{i-1}}\right)^\sigma(a)=0$.
By Lemma~\ref{lemmaderivadacomposta},
$$
\eta^{\sigma \Delta^{i-1}}(a) =(a_1)^{i-1} \left(\eta^{\Delta^{i-1}}\right)^\sigma(a),
$$
proving that $\eta^{\sigma\Delta^{i-1}}(a)=0$.
\end{proof}

We end this section recalling a result
that will be needed in the proof
of our Theorem~\ref{main:result}.

\begin{theorem}{\rm \cite{Serge:Lang}}
\label{Serge:Lang}
Let $S$ and $T$ be subsets of a  normed vector space. Let $f$ be a
map defined on $T \times S$, having values in some complete normed
vector space. Let $v$ be adherent to $S$ and $w$ adherent to $T$.
Assume that:
\begin{enumerate}
\item $\lim_{x\rightarrow v} f(t,x)$ exists for each $t \in T$;

\item $\lim_{t\rightarrow w} f(t,x)$ exists uniformly for  $x \in S$.
\end{enumerate}
Then the limits
$\lim_{t\rightarrow w} \lim_{x\rightarrow v} f(t,x)$,
$\lim_{x\rightarrow v} \lim_{t\rightarrow w}f(t,x)$,
and $\lim_{(t,x)\rightarrow (w,v)} f(t,x)$
all exist and are equal.
\end{theorem}


\section{Main Results}
\label{sec:mr}

We prove a first-order necessary optimality condition
for higher-order infinite horizon variational problems on time scales.
For simplicity of expressions, we introduce the operator
$\langle\cdot\rangle^r$ defined by
$\langle x\rangle^r(t):=
\left(t,x^{\sigma^{r}}(t), x^{\sigma^{r-1}\Delta}(t),
\ldots, x^{\sigma\Delta^{r-1}}(t), x^{\Delta^{r}}(t)\right)$.

\begin{theorem}[Euler--Lagrange Equation and Transversality Conditions]
\label{main:result}
Let $\mathbb{T}$ be a time scale satisfying condition $(H)$
and such that $\sup \mathbb{T}=+\infty$.
Suppose that $x_{\ast}$ be a maximizer to problem
\eqref{problem} and let $\eta \in C^{2r}_{rd}([a,+\infty[,\mathbb{R})$
be such that $\eta(a)=0$, $\eta^{\Delta}(a)=0$, \ldots,
$\eta^{\Delta^{r-1}}(a)=0$. Define
\begin{equation*}
\begin{split}
 A(\varepsilon, T^\prime) &:= \int_{a}^{T^\prime}
\frac{L\langle x_{\ast} + \epsilon\eta \rangle^r(t)
- L\langle x_{\ast}\rangle^r(t)}{\epsilon} \Delta t,\\
V(\varepsilon, T) &:= \inf_{T^\prime \geq
T}\int_{a}^{T^\prime} \Big(L\langle x_{\ast} + \epsilon\eta \rangle^r(t)
- L\langle x_{\ast}\rangle^r(t)\Big) \Delta t,\\
V(\varepsilon)&:= \lim_{T\rightarrow+\infty} V(\varepsilon, T).
\end{split}
\end{equation*}
Suppose that
\begin{enumerate}

\item $\displaystyle \lim_{\varepsilon \rightarrow 0}
\frac{V(\varepsilon, T) }{\varepsilon}$ exists for all $T$;

\item $\displaystyle \lim_{T\rightarrow+\infty}
\frac{V(\varepsilon, T) }{\varepsilon}$ exists uniformly for $\varepsilon$;

\item For every $T^\prime > a$, $T > a$,
and $\varepsilon\in \mathbb{R}\setminus\{0\}$,
there exists a sequence $\left(A(\varepsilon,
T^\prime_n)\right)_{n \in \mathbb{N}}$ such that
$$
\displaystyle \lim_{n \rightarrow +\infty} A(\varepsilon, T^\prime_n)
= \displaystyle \inf_{T^\prime \geq T} A(\varepsilon, T^\prime)
$$
uniformly for $\varepsilon$.
\end{enumerate}
Then $x_{\ast}$ satisfies the Euler--Lagrange equation
\begin{equation}
\label{E-L-equation}
\sum_{i=0}^{r} (-1)^i
\left(\frac{1}{a_1}\right)^{\frac{i(i-1)}{2}}\left(\partial_{i+2} L\right)^{\Delta^i}
\langle x \rangle^r(t) =0
\end{equation}
for all $t \in [a,+\infty[$ and the $r$ transversality conditions
\begin{multline}
\label{tranversality}
\displaystyle \lim_{T\rightarrow+\infty} \inf_{T^\prime \geq T}
\Biggl\{\left( \partial_{r+2-(k-1)} L\langle x \rangle^r(T^\prime) + \sum_{i=1}^{k-1} (-1)^{i}
\left(\partial_{r+2-(k-1)+i} L\right)^{\Delta^i} \langle x \rangle^r(T^\prime)
\cdot \Psi_i^r(k)\right) \\
\times x^{\sigma^{k-1}\Delta^{r-k}}(T^\prime)\Biggr\} =0,
\end{multline}
$k=1,\ldots,r$, with
$\displaystyle \Psi_i^r(k) = \prod_{j=1}^{i}\left(\frac{1}{a_1}\right)^{r-(k-1)+(j-1)}$.
\end{theorem}

The proof of Theorem~\ref{main:result}
is given in \S\ref{E-L_and_Transversality}.
Before that we state and prove several useful auxiliary results.
In particular, we  prove in \S\ref{Fundamental_Lemmas}
a higher-order integration by parts formula
(Lemma~\ref{integration-parts-higher-order}) and
three higher-order fundamental lemmas of the calculus of variations
on time scales (Lemmas~\ref{Fund.Lemma.1},
\ref{first-transversality} and \ref{Fund.Lemma.3}).


\subsection{Fundamental Lemmas}
\label{Fundamental_Lemmas}

In our results we use the standard convention
that $\sum_{k=1}^{j} \gamma(k) = 0$ whenever $j=0$.

\begin{lemma}[Higher-order integration by parts formula]
\label{integration-parts-higher-order}
Let $r \in \mathbb{N}$, $\mathbb{T}$ be a time scale satisfying condition
$(H)$, $a,b\in \mathbb{T}$, $a<b$, $f\in C^{r}_{rd}([a,\sigma^r(b)],\mathbb{R})$,
and $g\in C^{2r}_{rd}([a,\sigma^r(b)],\mathbb{R})$.
For each $i=1, \ldots, r$ we have
\begin{multline*}
\int_a^b f(t) g^{\sigma^{r-i}\Delta^{i}}(t) \Delta t
= \left[ f(t) g^{\sigma^{r-i}\Delta^{i-1}} (t)
+ \sum_{k=1}^{i-1}(-1)^k f^{\Delta^k}(t) g^{\sigma^{r-i+k}\Delta^{i-1-k}}(t)
\cdot \prod_{j=1}^{k}\left(\frac{1}{a_1}\right)^{i-j}\right]_{a}^{b}\\
+ (-1)^i \int_a^b \left(\frac{1}{a_1}\right)^{\frac{i(i-1)}{2}}
f^{\Delta^{i}} (t)  g^{\sigma^{r}}(t)\Delta t.
\end{multline*}
\end{lemma}

\begin{proof}
We prove the lemma by mathematical induction.
If $r=1$, the result is obviously true: it coincides
with the usual integration by parts formula on time scales.
Assuming that the result holds for an arbitrary $r$, we will prove it for $r+1$.
Fix $i=1,\ldots, r$. By the induction hypotheses,
\begin{equation*}
\begin{split}
\int_a^b f(t) & g^{\sigma^{r+1-i}\Delta^{i}}(t) \Delta t
= \int_a^b f(t)  (g^\sigma)^{\sigma^{r-i}\Delta^{i}}(t) \Delta t\\
&= \left[f(t) (g^\sigma)^{\sigma^{r-i}\Delta^{i-1}}(t)
+ \sum_{k=1}^{i-1}(-1)^k f^{\Delta^k}(t)  (g^\sigma)^{\sigma^{r-i+k}\Delta^{i-1-k}}(t)
\cdot \prod_{j=1}^{k}\left(\frac{1}{a_1}\right)^{i-j}\right]_{a}^{b}\\
& \qquad + (-1)^i  \int_a^b \left(\frac{1}{a_1}\right)^{\frac{i(i-1)}{2}}
f^{\Delta^{i}} (t)  (g^\sigma)^{\sigma^{r}}(t)\Delta t\\
&= \left[f(t) g^{\sigma^{r+1-i}\Delta^{i-1}}(t)
+ \sum_{k=1}^{i-1}(-1)^k f^{\Delta^k}(t)
g^{\sigma^{r+1-i+k}\Delta^{i-1-k}}(t) \cdot \prod_{j=1}^{k}\left(\frac{1}{a_1}\right)^{i-j}\right]_{a}^{b}\\
& \qquad + (-1)^i \displaystyle \int_a^b \left(\frac{1}{a_1}\right)^{\frac{i(i-1)}{2}}
f^{\Delta^{i}} (t) g^{\sigma^{r+1}}(t)\Delta t.
\end{split}
\end{equation*}
It remains to prove that the result is true for $i=r+1$. Note that
\begin{equation*}
\begin{split}
\int_a^b & f(t) g^{\Delta^{r+1}}(t) \Delta t
= \int_a^b f(t)  (g^\Delta)^{\Delta^{r}}(t) \Delta t\\
&= \left[f(t) (g^\Delta)^{\Delta^{r-1}} (t) +
\sum_{k=1}^{r-1}(-1)^k f^{\Delta^k}(t)  (g^\Delta)^{\sigma^{k}\Delta^{r-1-k}}(t)
\cdot \prod_{j=1}^{k}\left(\frac{1}{a_1}\right)^{r-j}\right]_{a}^{b}\\
&\qquad + (-1)^r \int_a^b \left(\frac{1}{a_1}\right)^{\frac{r(r-1)}{2}}
f^{\Delta^{r}}(t) (g^\Delta)^{\sigma^{r}}(t)\Delta t
\ \quad \quad (\mbox {by induction hypotheses})\\
&= \left[f(t) g^{\Delta^{r}} (t) +
\sum_{k=1}^{r-1}(-1)^k f^{\Delta^k}(t)
g^{\sigma^{k}\Delta^{r-k}}(t) \cdot \left( \frac{1}{a_1}\right)^{k}
\prod_{j=1}^{k}\left(\frac{1}{a_1}\right)^{r-j}\right]_{a}^{b}\\
& \qquad + \ (-1)^r \displaystyle \int_a^b \left(\frac{1}{a_1}\right)^{\frac{r(r-1)}{2}}
\left( \frac{1}{a_1}\right)^{r} f^{\Delta^{r}} (t) (g^{\sigma^{r}})^{\Delta}(t)\Delta t
\ \quad \quad (\mbox {by Lemma} \  \ref{lemmaderivadacomposta}).
\end{split}
\end{equation*}
Using the standard integration by parts formula
in the last delta integral, and taking into account that
$\left( \frac{1}{a_1}\right)^{k} \prod_{j=1}^{k}\left(\frac{1}{a_1}\right)^{r-j}
= \prod_{j=1}^{k}\left(\frac{1}{a_1}\right)^{r+1-j}$,
we conclude that
\begin{equation*}
\begin{split}
\int_a^b & f(t) g^{\Delta^{r+1}}(t) \Delta t
= \left[f(t) g^{\Delta^{r}}(t)
+ \sum_{k=1}^{r-1}(-1)^k f^{\Delta^k}(t)  g^{\sigma^{k}\Delta^{r-k}}(t)
\cdot \prod_{j=1}^{k}\left(\frac{1}{a_1}\right)^{r+1-j}\right]_{a}^{b}\\
& \qquad + \left[(-1)^r f^{\Delta^{r}} (t) g^{\sigma^{r}}(t)
\left( \frac{1}{a_1}\right)^{\frac{r(r+1)}{2}} \right]_{a}^{b}
- (-1)^{r} \int_a^b \left(\frac{1}{a_1}\right)^{\frac{r(r+1)}{2}}
f^{\Delta^{r+1}} (t)  g^{\sigma^{r+1}}(t)\Delta t \\
&= \left[f(t) g^{\Delta^{r}}(t)
+ \sum_{k=1}^{r}(-1)^k f^{\Delta^k}(t)  g^{\sigma^{k}\Delta^{r-k}}(t)
\cdot \prod_{j=1}^{k}\left(\frac{1}{a_1}\right)^{r+1-j}\right]_{a}^{b}\\
& \qquad + (-1)^{r+1} \int_a^b \left(\frac{1}{a_1}\right)^{\frac{r(r+1)}{2}}
f^{\Delta^{r+1}}(t) g^{\sigma^{r+1}}(t)\Delta t,
\end{split}
\end{equation*}
proving that the result is true for $i=r+1$.
\end{proof}

Before presenting the higher-order fundamental lemmas of the calculus
of variations on time scales, we need the following three preliminary results.

\begin{lemma}
\label{teorema_tecnico}
Suppose that $a \in \mathbb{T}$ and
$f\in C_{rd}([a,+\infty[, \mathbb{R})$
be such that $f \geq 0$ on $[a,+\infty[$. If
$$\lim_{T\rightarrow+\infty}\inf_{T^\prime \geq
T} \int_{a}^{T^\prime}f(t)\Delta t=0,$$
then $f \equiv 0$ on $[a,+\infty[$.
\end{lemma}

\begin{proof}
Suppose, by contradiction, that there exists $t_0\in [a,+\infty[$ such
that $f(t_0)>0$. Fix $b \in \mathbb{T}$ such that $a_0\leq t_0 <b$. We will prove that
$\int_a^b f(t)\Delta t>0$.
If $t_0$ is right-scattered, then
\begin{equation*}
\begin{split}
\int_{a}^{b}f(t) \Delta t
&= \int_{a}^{t_0}f(t)\Delta t
+ \int_{t_0}^{\sigma(t_0)}f(t)\Delta t
+ \int_{\sigma(t_0)}^{b}f(t)\Delta t\\
&\geq \int_{t_0}^{\sigma(t_0)}f(t)\Delta t
=f(t_0)\left(\sigma(t_0)-t_0\right)>0.
\end{split}
\end{equation*}
If $t_0$ is right-dense, then,
by the continuity of $f$ at $t_0$,
there exists $\delta>0$ such that $f(t)>0$
for all $t\in[t_0,t_0+\delta[$ and, therefore,
\begin{equation*}
\begin{split}
\int_{a}^{b} f(t) \Delta t
&= \int_{a}^{t_0}f(t) \Delta t
+ \int_{t_0}^{t_0+\delta}f(t) \Delta t
+ \int_{t_0+\delta}^{b}f(t) \Delta t\\
&\geq \int_{t_0}^{t_0+\delta}f(t)\Delta t > 0.
\end{split}
\end{equation*}
Then, for any $T>b$,
$$
\inf_{T^\prime \geq T} \int_{a}^{T^\prime}f(t)\Delta t=
\int_{a}^{T}f(t) \Delta t\geq \int_{a}^{b}f(t) \Delta t >0,
$$
and hence
$$
\lim_{T\rightarrow+\infty} \inf_{T^\prime \geq T}
\int_{a}^{T^\prime}f(t)\Delta \geq \int_{a}^{b}f(t) \Delta t >0,
$$
which is a contradiction.
\end{proof}

\begin{lemma}
\label{lemma2}
Let $f \in C_{rd}([a,+\infty[, \mathbb{R})$. If
$$
\lim_{T\rightarrow+\infty}\inf_{T^\prime \geq T}
\int_{a}^{T^\prime}f(t)\eta^{\Delta}(t)\Delta t=0
$$
for all $\eta \in C_{rd}^1([a,+\infty[, \mathbb{R})$
such that $\eta(a)=0$, then
$f(t)= c$  for all $t\in [a,+\infty[$, where $c \in \mathbb{R}$.
\end{lemma}

\begin{proof}
Fix $T, T^\prime \in \mathbb{T}$ such that $T^\prime \geq T >a$.
Let $c$ be a constant defined by the condition
$$
\int_{a}^{T^\prime}\left(f(\tau)-c\right)\Delta\tau=0,
$$
and let
$$
\eta(t)=\int_{a}^{t}\left(f(\tau)-c\right)\Delta\tau.
$$
Clearly, $\eta \in C_{rd}^1([a,+\infty[, \mathbb{R})$,
$\eta^{\Delta}(t)=f(t)-c$,
$$
\eta(a)=\int_{a}^{a}\left(f(\tau)-c\right)\Delta\tau=0,
\quad \mbox{ and } \quad
\eta(T^\prime)=\int_{a}^{T^\prime}\left(f(\tau)-c\right)\Delta\tau=0.
$$
Observe that
$$
\int_{a}^{T^\prime}\left(f(t)-c\right)\eta^{\Delta}(t)\Delta t=
\int_{a}^{T^\prime}\left(f(t)-c\right)^2\Delta t
$$
and
$$
\int_{a}^{T^\prime}\left(f(t)-c\right)\eta^{\Delta}(t)\Delta t
= \int_{a}^{T^\prime}f(t)\eta^{\Delta}(t)\Delta t- c
\int_{a}^{T^\prime}\eta^{\Delta}(t)\Delta t
=\int_{a}^{T^\prime}f(t)\eta^{\Delta}(t)\Delta t.
$$
Hence,
$$
\lim_{T\rightarrow+\infty}\inf_{T^\prime \geq T}
\int_{a}^{T^\prime}f(t)\eta^{\Delta}(t)\Delta t
= \lim_{T\rightarrow+\infty}\inf_{T^\prime \geq T}
\int_{a}^{T^\prime}\left(f(t)-c\right)^2\Delta t=0,
$$
which shows, by Lemma~\ref{teorema_tecnico}, that
$f(t)-c=0$ for all $t \in [a,+\infty[$.
\end{proof}

\begin{lemma}
\label{lemma4}
Let $f, g \in C_{rd}([a,+\infty[, \mathbb{R})$. If
$$
\lim_{T\rightarrow+\infty}\inf_{T^\prime \geq
T} \int_{a}^{T^\prime}\left(f(t)\eta^\sigma(t) + g(t)
\eta^{\Delta}(t)\right)\Delta t=0
$$
for all $\eta \in C_{rd}^1([a,+\infty[, \mathbb{R})$
such that $\eta(a)=0$, then $g$ is delta differentiable and
$$
g^{\Delta}(t)=f(t) \quad \forall t\in [a,+\infty[.
$$
\end{lemma}

\begin{proof}
Fix $T, T^\prime \in \mathbb{T}$ such that
$T^\prime\geq T>a$ and
define $A(t)=\int_{a}^{t} f(\tau)\Delta \tau$. Then
$A^{\Delta} (t)=f(t)$ for all $t \in [a,+\infty[$  and
$$
\int_{a}^{T^\prime} A(t)\eta^{\Delta}(t) \Delta t
= \displaystyle \left[ A(t)\eta(t) \right]_{a}^{T^\prime}
- \int_{a}^{T^\prime}A^{\Delta}(t) \eta^\sigma (t) \Delta t
=A(T^\prime)\eta(T^\prime) - \int_{a}^{T^\prime}
f(t) \eta^\sigma(t) \Delta t.
$$
Restricting $\eta$ to those such that $\eta(T^\prime)=0$, we obtain
$$
\lim_{T\rightarrow+\infty}\inf_{T^\prime \geq
T} \int_{a}^{T^\prime} \left(f(t)\eta^\sigma(t) + g(t)
\eta^{\Delta}(t)\right)\Delta t =
\lim_{T\rightarrow+\infty}\inf_{T^\prime \geq
T} \int_{a}^{T^\prime} \left(-A(t) + g(t) \right)\eta^{\Delta}(t)\Delta t = 0.
$$
By Lemma~\ref{lemma2} we may conclude that there exists
$c \in \mathbb{R}$ such that $-A(t) + g(t)=c$
for all $t \in  [a,+\infty[$. Therefore,
$A^{\Delta}(t)=g^{\Delta}(t)$ for all $t \in [a,+\infty[$,
proving the desired result:
$g^{\Delta}(t)=f(t)$ for all $ t\in [a,+\infty[$.
\end{proof}

We are now in conditions to prove the following three
fundamental lemmas of the calculus of variations for higher-order
infinite horizon variational problems on time scales.

\begin{lemma}[Higher-order fundamental lemma of the calculus of variations I]
\label{Fund.Lemma.1}
Let $\mathbb{T}$ be a time scale satisfying condition $(H)$
and such that $\sup \mathbb{T}=+\infty$.
Suppose that $f_0 \in C_{rd}([a,+\infty[, \mathbb{R})$,
$f_1\in C^{1}_{rd}([a,+\infty[, \mathbb{R})$,
$\ldots$, $f_r \in C^{r}_{rd}([a,+\infty[, \mathbb{R})$. If
$$
\lim_{T\rightarrow+\infty}\inf_{T^\prime \geq
T} \int_{a}^{T^\prime}
\left(\sum_{i=0}^{r}f_i(t) \eta^{\sigma^{r-i}\Delta^{i}}(t) \right)
\Delta t=0
$$
for all $\eta \in C_{rd}^{2r}([a, +\infty[, \mathbb{R})$ such that
$\eta(a)=0$, $\eta^{\Delta}(a)=0$, \ldots, $\eta^{\Delta^{r-1}}(a)=0$, then
\begin{equation*}
\sum_{i=0}^{r} (-1)^i
\left(\frac{1}{a_1}\right)^{\frac{i(i-1)}{2}}f_i^{\Delta^i}(t) =0
\quad \forall t \in [a,+\infty[.
\end{equation*}
\end{lemma}

\begin{proof}
We prove the lemma by mathematical induction.
If $r=1$, the result is true by Lemma~\ref{lemma4}.
Assume now that the result is true for some $r$.
We will prove that the result is also true for $r+1$.
Suppose that
$$
\lim_{T\rightarrow+\infty}\inf_{T^\prime \geq
T} \int_{a}^{T^\prime}
\left(\sum_{i=0}^{r+1}f_i(t) \eta^{\sigma^{r+1-i}\Delta^{i}}(t)
\right) \Delta t=0
$$
for all $\eta \in C_{rd}^{2(r+1)}\left([a, +\infty[,\mathbb{R}\right)$
such that $\eta(a)=0$, $\eta^{\Delta}(a)=0$, \ldots, $\eta^{\Delta^{r}}(a)=0$.
We want to prove that
$$
\sum_{i=0}^{r+1} (-1)^i
\left(\frac{1}{a_1}\right)^{\frac{i(i-1)}{2}}f_i^{\Delta^i}(t) =0
\quad \forall t \in [a,+\infty[.
$$
Note that
$$
\int_{a}^{T^\prime}
\left(\sum_{i=0}^{r+1}f_i(t) \eta^{\sigma^{r+1-i}\Delta^{i}}(t)
\right) \Delta t = \int_{a}^{T^\prime}
\left(\sum_{i=0}^{r}f_i(t) \eta^{\sigma^{r+1-i}\Delta^{i}}(t)
\right) \Delta t \ +  \int_{a}^{T^\prime}  f_{r+1}(t)
\left(\eta^{\Delta^r}\right)^{\Delta}(t) \Delta t.
$$
Using the integration by parts formula in the last integral, we obtain that
$$
\int_{a}^{T^\prime}  f_{r+1}(t) \left (\eta^{\Delta^r}
\right)^{\Delta}(t)
\Delta t = \left[f_{r+1}(t)\eta^{\Delta^r}(t)
\right]^{T^\prime}_{a} - \int_{a}^{T^\prime}
f_{r+1}^{\Delta}(t)\left (\eta^{\Delta^r} \right)^{\sigma}(t) \Delta t.
$$
Since $\eta^{\Delta^{r}}\left(a\right)=0$
and we can restrict ourselves to those $\eta$ such that
$\eta^{\Delta^{r}}\left(T^\prime\right)=0$, then
$$
\int_{a}^{T^\prime}  f_{r+1}(t) \left (\eta^{\Delta^r}
\right)^{\Delta}(t)
\Delta t  = - \int_{a}^{T^\prime} f_{r+1}^{\Delta}(t)\left
(\eta^{\Delta^r} \right)^{\sigma}(t) \Delta t
$$
and, by Lemma~\ref{lemmaderivadacomposta},
$$
\int_{a}^{T^\prime}  f_{r+1}(t) \left (\eta^{\Delta^r}
\right)^{\Delta}(t)
\Delta t  = - \int_{a}^{T^\prime} f_{r+1}^{\Delta}(t)\left(
\frac{1}{a_1}\right)^{r}\eta^{\sigma \Delta^r}(t) \Delta t.
$$
Hence,
\begin{equation*}
\begin{split}
\int_{a}^{T^\prime} & \left(
\sum_{i=0}^{r+1}f_i(t)
\eta^{\sigma^{r+1-i}\Delta^{i}}(t) \right) \Delta t \\
&= \int_{a}^{T^\prime}
\left(\sum_{i=0}^{r}f_i(t) \eta^{\sigma^{r+1-i}\Delta^{i}}(t)\right)
\Delta t -  \int_{a}^{T^\prime} f_{r+1}^{\Delta}(t)\left(
\frac{1}{a_1}\right)^{r}\eta^{\sigma \Delta^r}(t) \Delta t\\
&= \int_{a}^{T^\prime}
\left(\sum_{i=0}^{r-1}f_i(t) \left(\eta^\sigma
\right)^{\sigma^{r-i}\Delta^{i}}(t)
+ \left( f_r(t) - f_{r+1}^{\Delta}(t)\left(
\frac{1}{a_1}\right)^{r}\right)(\eta^\sigma)^{\Delta^r}(t)
\right) \Delta t
\end{split}
\end{equation*}
and, therefore,
\begin{multline*}
\lim_{T\rightarrow+\infty}\inf_{T^\prime \geq T}
\int_{a}^{T^\prime} \left(
\sum_{i=0}^{r+1}f_i(t)
\eta^{\sigma^{r+1-i}\Delta^{i}}(t) \right) \Delta t \\
= \lim_{T\rightarrow+\infty}\inf_{T^\prime
\geq T} \int_{a}^{T^\prime}
\left[\sum_{i=0}^{r-1}f_i(t) \left(\eta^\sigma
\right)^{\sigma^{r-i}\Delta^{i}}(t)
 + \left( f_r(t) - f_{r+1}^{\Delta}(t)\left(
\frac{1}{a_1}\right)^{r}\right) (\eta^\sigma)^{\Delta^r}(t)
 \right] \Delta t=0.
\end{multline*}
By Lemma~\ref{lemma_funcoes_admissiveis_1},
$\eta^\sigma(a) = 0$, $(\eta^\sigma)^{\Delta}(a)= 0$,
\ldots, $(\eta^\sigma)^{\Delta^{r-1}}(a)=0$.
Then, by the induction hypothesis, we conclude that
$$
\sum_{i=0}^{r-1} (-1)^i
\left(\frac{1}{a_1}\right)^{\frac{i(i-1)}{2}}f_i^{\Delta^i}(t)
+ (-1)^r \left(\frac{1}{a_1}\right)^{\frac{r(r-1)}{2}}
\left(f_r(t)
- f^{\Delta}_{r+1}(t)\left(\frac{1}{a_1}\right)^r\right)^{\Delta^r}(t) =0
$$
for all $t \in [a,+\infty[$, which is equivalent to
$$
\sum_{i=0}^{r+1} (-1)^i
\left(\frac{1}{a_1}\right)^{\frac{i(i-1)}{2}}f_i^{\Delta^i}(t) =0
\quad \forall  t \in [a,+\infty[.
$$
\end{proof}

\begin{lemma}[Higher-order fundamental lemma of the calculus of variations II]
\label{first-transversality}
Let $\mathbb{T}$ be a time scale satisfying
condition $(H)$ and such that $\sup \mathbb{T}=+\infty$.
Suppose that  $f_0 \in C_{rd}([a,+\infty[, \mathbb{R})$ and
$f_i \in C^{r}_{rd}([a,+\infty[, \mathbb{R})$ for all $i=1, \ldots, r$. If
$$
\lim_{T\rightarrow+\infty}  \inf_{T^\prime \geq
T} \displaystyle \int_{a}^{T^\prime}  \left(\sum_{i=0}^{r}f_i(t)
\eta^{\sigma^{r-i} \Delta^{i}}(t) \right)\Delta t=0
$$
for all $\eta \in C_{rd}^{2r}([a, +\infty[, \mathbb{R})$ such that
$\eta(a)=0$, $\eta^{\Delta}(a)=0$, \ldots, $\eta^{\Delta^{r-1}}(a)=0$, then
$$
\lim_{T\rightarrow+\infty} \inf_{T^\prime \geq T}
\left\{f_r(T^\prime)\cdot \eta^{\Delta^{r-1}}(T^\prime)\right\} =0.
$$
\end{lemma}

\begin{proof}
Note that
\begin{equation*}
\begin{split}
\int_{a}^{T^\prime} & \left(\sum_{i=0}^{r} f_i(t)
\eta^{\sigma^{r-i}\Delta^{i}}(t) \Delta t\right) \\
&= \int_{a}^{T^\prime} f_0(t) \eta^{\sigma^r}(t)\Delta t
+ \sum_{i=1}^{r}\left( \int_{a}^{T^\prime} f_i(t)
\eta^{\sigma^{r-i}\Delta^{i}} (t) \Delta t\right)\\
&= \int_{a}^{T^\prime} f_0(t) \eta^{\sigma^r}(t)\Delta t\\
&\qquad + \sum_{i=1}^{r} \left[f_i(t)\eta^{\sigma^{r-i}\Delta^{i-1}} (t)
+ \sum_{k=1}^{i-1}(-1)^k f_i^{\Delta^k}(t) \eta^{\sigma^{r-i+k}\Delta^{i-1-k}}(t)
\cdot \prod_{j=1}^{k}\left(\frac{1}{a_1}\right)^{i-j}\right]_{a}^{T^\prime}\\
& \qquad + \sum_{i=1}^{r} \left((-1)^i
\int_{a}^{T^\prime} \left(\frac{1}{a_1}\right)^{\frac{i(i-1)}{2}}
f_i^{\Delta^i}(t) \eta^{\sigma^r}(t) \Delta t \right)
\quad \quad (\mbox{by Lemma~\ref{integration-parts-higher-order}})\\
&= \int_{a}^{T^\prime} \left(f_0(t)
+ \sum_{i=1}^{r} (-1)^i
f^{^{\Delta^i}}_i(t)\left(\frac{1}{a_1}\right)^{\frac{i(i-1)}{2}}\right)
\cdot \eta^{\sigma^r}(t) \Delta t \\
& \qquad + \sum_{i=1}^{r} \left[ \left(f_i(t)\eta^{\sigma^{r-i}\Delta^{i-1}}(t)
+ \sum_{k=1}^{i-1} (-1)^k f_i^{\Delta^k}(t) \eta^{\sigma^{r-i+k}\Delta^{i-1-k}}(t)
\cdot \prod_{j=1}^{k}\left(\frac{1}{a_1}\right)^{i-j}\right)\right]_{a}^{T^\prime}\\
&= \int_{a}^{T^\prime} \left( \sum_{i=0}^{r} (-1)^i f^{^{\Delta^i}}_i(t)\left(
\frac{1}{a_1}\right)^{\frac{i(i-1)}{2}}\right)\cdot \eta^{\sigma^r}(t) \Delta t \\
& \qquad + \quad \sum_{i=1}^{r-1} \left[ \left(f_i(t)\eta^{\sigma^{r-i}\Delta^{i-1}}(t)
+ \sum_{k=1}^{i-1} (-1)^k f_i^{\Delta^k}(t) \eta^{\sigma^{r-i+k}\Delta^{i-1-k}}(t)
\cdot \prod_{j=1}^{k}\left(\frac{1}{a_1}\right)^{i-j} \right)\right]_{a}^{T^\prime}\\
& \qquad + \left [ f_r(t)\eta^{\Delta^{r-1}}(t) +
\sum_{k=1}^{r-1} (-1)^k f_r^{\Delta^k}(t)
\eta^{\sigma^{k}\Delta^{r-1-k}}(t) \cdot \prod_{j=1}^{k}\left(
\frac{1}{a_1}\right)^{r-j} \right ]_{a}^{T^\prime}
\end{split}
\end{equation*}
and, by Lemma~\ref{Fund.Lemma.1}, we get
\begin{equation*}
\begin{split}
\int_{a}^{T^\prime} & \left(\sum_{i=0}^{r}
f_i(t) \eta^{\sigma^{r-i}\Delta^{i}}(t) \Delta t\right)\\
&= \sum_{i=1}^{r-1} \left[ \left(f_i(t)\eta^{\sigma^{r-i}\Delta^{i-1}}(t)
+ \sum_{k=1}^{i-1} (-1)^k f_i^{\Delta^k}(t) \eta^{\sigma^{r-i+k}\Delta^{i-1-k}}(t)
\cdot \prod_{j=1}^{k}\left(\frac{1}{a_1}\right)^{i-j} \right)\right]_{a}^{T^\prime}\\
& \qquad + \left [ f_r(t)\eta^{\Delta^{r-1}}(t)
+ \sum_{k=1}^{r-1} (-1)^k f_r^{\Delta^k}(t)
\eta^{\sigma^{k}\Delta^{r-1-k}}(t) \cdot \prod_{j=1}^{k}\left(
\frac{1}{a_1}\right)^{r-j} \right ]_{a}^{T^\prime}.
\end{split}
\end{equation*}
Therefore, restricting the variations $\eta$ to those such that
$$
\eta^{\sigma^{r-k}\Delta^{k-1}}(T^\prime)
=\eta^{\sigma^{r-k}\Delta^{k-1}}(a)=0, \quad k=1,\ldots, r-1,
$$
$$
\eta^{\sigma^{k}\Delta^{r-1-k}}(T^\prime)
=\eta^{\sigma^{k}\Delta^{r-1-k}}(a)=0, \quad k=1,\ldots, r-1,
$$
we get
$$
\lim_{T\rightarrow+\infty}
\inf_{T^\prime \geq T} \int_{a}^{T^\prime}
\left(\sum_{i=0}^{r}f_i(t) \eta^{\sigma^{r-i}\Delta^{i}}(t)
\right) \Delta t=0
\Rightarrow
\lim_{T\rightarrow+\infty}\inf_{T^\prime \geq T}
\left\{ f_r(T^\prime) \eta^{\Delta^{r-1}} (T^\prime) \right\}=0,
$$
proving the desired result.
\end{proof}

\begin{lemma}[Higher-order fundamental lemma of the calculus of variations III]
\label{Fund.Lemma.3}
Let $\mathbb{T}$ be a time scale satisfying condition $(H)$ and such that $\sup \mathbb{T}=+\infty$.
Suppose that $f_0 \in C_{rd}([a,+\infty[, \mathbb{R})$ and
$f_i \in C^{r}_{rd}([a,+\infty[, \mathbb{R})$ for all $i=1, \ldots, r$. If
$$
\lim_{T\rightarrow+\infty}  \inf_{T^\prime \geq
T} \displaystyle \int_{a}^{T^\prime}  \left(\sum_{i=0}^{r}f_i(t)
\eta^{\sigma^{r-i} \Delta^{i}}(t) \right)\Delta t=0
$$
for all $\eta \in C_{rd}^{2r}([a, +\infty[, \mathbb{R})$ such that
$\eta(a)=0$, $\eta^{\Delta}(a)=0$,
\ldots, $\eta^{\Delta^{r-1}}(a)=0$, then
$$
\lim_{T\rightarrow+\infty} \inf_{T^\prime \geq T}
\left\{\left(f_{r-(k-1)} (T^\prime) + \sum_{i=1}^{k-1}
(-1)^{i} \left(f_{r-(k-1)+i} \right)^{\Delta^i} (T^\prime)
\cdot \Psi_i^r(k) \right)
\cdot \eta^{\sigma^{k-1}\Delta^{r-k}}(T^\prime)\right\} =0
$$
for all $k=1,\ldots,r$, where
$\displaystyle \Psi_i^r(k) = \prod_{j=1}^{i}\left(\frac{1}{a_1}\right)^{r-(k-1)+(j-1)}$.
\end{lemma}

\begin{proof}
We do the proof by induction. Let $r=1$.
Using the integration by parts formula
and Lemma~\ref{Fund.Lemma.1}, we obtain
$\lim_{T\rightarrow+\infty} \inf_{T^\prime \geq T}
f_1(T^\prime) \eta(T^\prime)=0$,
showing that the result is true for $r=1$.
Assuming now that the result holds for an arbitrary $r$,
we will prove it for $r+1$. Suppose that
$$
\lim_{T\rightarrow+\infty}\inf_{T^\prime \geq
T} \int_{a}^{T^\prime}
\left(\sum_{i=0}^{r+1}f_i(t) \eta^{\sigma^{r+1-i}\Delta^{i}}(t)
\right) \Delta t=0
$$
for all $\eta \in C_{rd}^{2(r+1)}([a, +\infty[,\mathbb{R})$ such that
$\eta^{\Delta^{i}}(a)=0$, $i=0, \ldots, r$. We want to prove that
\begin{multline}
\label{tese-de-inducao}
\lim_{T\rightarrow+\infty} \inf_{T^\prime \geq T}
\Biggl\{\left (f_{r+1-(k-1)} (T^\prime) + \sum_{i=1}^{k-1} (-1)^{i}
\left(f_{r+1-(k-1)+i} \right)^{\Delta^i} (T^\prime)
\cdot \Psi_i^{r+1}(k) \right)\\
\times \eta^{\sigma^{k-1}\Delta^{r+1-k}}(T^\prime)\Biggr\} = 0
\end{multline}
for $k=1, \ldots, r+1$. Fix some $k=2,\ldots, r+1$. The main idea of the proof
is that the $k$th transversality condition for the variational
problem of order $r+1$ is obtained from the $(k-1)$th transversality
condition for the variational problem of order $r$.
Using the same techniques as in Lemma~\ref{Fund.Lemma.1}, we prove that
$$
\lim_{T\rightarrow+\infty}\inf_{T^\prime \geq
T} \int_{a}^{T^\prime}
\left(\sum_{i=0}^{r+1}f_i(t) \eta^{\sigma^{r+1-i}\Delta^{i}}(t)
\right) \Delta t=0
$$
implies
$$
\lim_{T\rightarrow+\infty}\inf_{T^\prime \geq
T} \left\{\int_{a}^{T^\prime}
\left(\sum_{i=0}^{r-1}f_i(t) \left(\eta^\sigma
\right)^{\sigma^{r-i}\Delta^{i}}(t)
 + \left( f_r(t) - f_{r+1}^{\Delta}(t)\left(
\frac{1}{a_1}\right)^{r}\right) (\eta^\sigma)^{\Delta^r}(t)
 \right) \Delta t\right\}=0.
$$
Since, by Lemma~\ref{lemma_funcoes_admissiveis_1},
$\eta^\sigma(a) = 0$, $(\eta^\sigma)^{\Delta}(a)= 0$,
\ldots, $(\eta^\sigma)^{\Delta^{r-1}}(a)=0$, then,
by the induction hypothesis for $k-1$, we conclude that
\begin{multline*}
\lim_{T\rightarrow+\infty} \inf_{T^\prime \geq T}
\Bigg\{\Bigg(\sum_{i=1}^{k-3} (-1)^{i}
\Big(f_{r-(k-2)+i} \Big)^{\Delta^i} (T^\prime)\cdot \Psi_i^r(k-1)
+ (-1)^{k-2}\Big(f_{r} \Big)^{\Delta^{k-2}} (T^\prime) \cdot \Psi_{k-2}^r(k-1)\\
+ f_{r-(k-2)} (T^\prime) + (-1)^{k-1}\Big(f_{r+1} \Big)^{\Delta^{k-1}} (T^\prime)
\cdot \Psi_{k-2}^{r}(k-1) \Big(\frac{1}{a_1}\Big)^{r}\Bigg)
\cdot (\eta^\sigma)^{\sigma^{k-2}\Delta^{r-(k-1)}}(T^\prime)\Bigg\} = 0,
\end{multline*}
which is equivalent to
\begin{multline*}
\lim_{T\rightarrow+\infty} \inf_{T^\prime \geq T}
\Biggl\{\left(f_{r-(k-2)} (T^\prime)
+ \sum_{i=1}^{k-1} (-1)^{i} \left(f_{r-(k-2)+i} \right)^{\Delta^i} (T^\prime)
\cdot \Psi_i^r(k-1)\right)\\
\times \eta^{\sigma^{k-1}\Delta^{r-(k-1)}}(T^\prime)\Biggr\} = 0.
\end{multline*}
This proves equation \eqref{tese-de-inducao} for $k=2,\ldots, r+1$.
It remains to prove that equation \eqref{tese-de-inducao} is also true for $k=1$.
This condition follows from  Lemma~\ref{first-transversality}.
\end{proof}


\subsection{Proof of Theorem~\ref{main:result}}
\label{E-L_and_Transversality}

Using our notion of weak maximality, if $x_{\ast}$ is optimal, then
$V(\varepsilon) \leq 0$ for every $\varepsilon \in \mathbb{R}$.
Since $V(0)=0$, then 0 is an extremal point of $V$.
We will prove that $V$ is differentiable at $t=0$, and hence
$V^\prime(0)=0$. Note that
\begin{equation*}
\begin{split}
V^\prime(0) &=
\lim_{\varepsilon \rightarrow 0} \frac{V(\varepsilon)}{\varepsilon}
= \lim_{\varepsilon \rightarrow 0}
\lim_{T\rightarrow+\infty}\frac{V(\varepsilon, T) }{\varepsilon}\\
&= \lim_{T\rightarrow+\infty}
\lim_{\varepsilon \rightarrow 0} \frac{V(\varepsilon, T) }{\varepsilon}
\qquad (\mbox{by hypothesis \emph{1} and \emph{2} and Theorem \ref{Serge:Lang})}\\
&= \lim_{T\rightarrow+\infty}
\lim_{\varepsilon \rightarrow 0} \displaystyle \inf_{T^\prime \geq T} A(\varepsilon, T^\prime)\\
&= \lim_{T\rightarrow+\infty}
\lim_{\varepsilon \rightarrow 0} \lim_{n \rightarrow +\infty} A(\varepsilon, T^\prime_n)
\qquad (\mbox{by hypothesis \emph{3})}\\
&= \lim_{T\rightarrow+\infty}
\lim_{n \rightarrow +\infty} \lim_{\varepsilon \rightarrow 0} A(\varepsilon, T^\prime_n)
\qquad (\mbox{by hypothesis \emph{3} and Theorem \ref{Serge:Lang}})\\
&= \lim_{T\rightarrow+\infty}
\inf_{T^\prime \geq T} \lim_{\varepsilon \rightarrow 0} A(\varepsilon, T^\prime)
\qquad (\mbox{by hypothesis \emph{3}})\\
&= \lim_{T\rightarrow+\infty}
\inf_{T^\prime \geq T}  \lim_{\varepsilon \rightarrow 0}
\int_{a}^{T^\prime} \frac{L\langle x_{\ast} + \epsilon\eta \rangle^r(t)
- L\langle x_{\ast}\rangle^r(t)}{\varepsilon} \Delta t\\
&= \lim_{T\rightarrow+\infty}  \inf_{T^\prime \geq T}
\int_{a}^{T^\prime} \lim_{\varepsilon \rightarrow 0}
\frac{L\langle x_{\ast} + \epsilon\eta \rangle^r(t)
- L\langle x_{\ast}\rangle^r(t)}{\varepsilon} \Delta t\\
&= \lim_{T\rightarrow+\infty}  \inf_{T^\prime \geq T}
\int_{a}^{T^\prime}  \left(\sum_{i=0}^{r}\partial_{i+2}
L\langle x_{\ast}\rangle^r(t) \cdot \eta^{\sigma^{r-i} \Delta^{i}}(t) \right)\Delta t.
\end{split}
\end{equation*}
Therefore,
$$
\lim_{T\rightarrow+\infty}  \inf_{T^\prime \geq
T} \displaystyle \int_{a}^{T^\prime}  \left(\sum_{i=0}^{r}\partial_{i+2}
L\langle x_{\ast}\rangle^r(t)\cdot \eta^{\sigma^{r-i} \Delta^{i}}(t) \right)\Delta t=0.
$$
Using Lemma~\ref{Fund.Lemma.1}, we conclude that
\begin{equation*}
\sum_{i=0}^{r} (-1)^i
\left(\frac{1}{a_1}\right)^{\frac{i(i-1)}{2}}\left(
\partial_{i+2} L\right)^{\Delta^i}
\langle x_{\ast}\rangle^r(t) =0
\quad \forall t \in [a,+\infty[ \, ,
\end{equation*}
proving that $x_{\ast}$ satisfies
the Euler--Lagrange equation \eqref{E-L-equation}.
By Lemma~\ref{Fund.Lemma.3}, we conclude that
\begin{multline}
\label{tranversality_p}
\lim_{T\rightarrow+\infty} \inf_{T^\prime \geq T}
\Biggl\{\left(\partial_{r+2-(k-1)} L\langle x_{\ast}\rangle^r(T^\prime)
+ \sum_{i=1}^{k-1} (-1)^{i} \left(\partial_{r+2-(k-1)+i} L \right)^{\Delta^i}
\langle x_{\ast}\rangle^r(T^\prime) \cdot \Psi_i^r(k) \right)\\
\times \eta^{\sigma^{k-1}\Delta^{r-k}}(T^\prime)\Biggr\} =0
\end{multline}
for $k=1,\ldots, r$. Consider $\eta$ defined by
$\eta(t)=\alpha(t) x_{\ast}(t)$ for all $t \in [a, +\infty[$,
where $\alpha: [a, +\infty[ \rightarrow \mathbb{R}$ is a $C^{2r}_{rd}$
function satisfying $\alpha (a)=0, \alpha^{\Delta}(a)=0, \ldots,
\alpha^{\Delta^{r-1}}(a)=0$, and there exists $T_0\in \mathbb{T}$
such that $\alpha(t)=\beta \in \mathbb{R}\setminus\{0\}$ for all $t> T_0$.
Note that $\eta(a)=0, \eta^{\Delta}(a)=0, \ldots, \eta^{\Delta^{r-1}}(a)=0$.
Substituting $\eta$ in  equation \eqref{tranversality_p}, we conclude that
\begin{multline*}
\lim_{T\rightarrow+\infty} \inf_{T^\prime \geq T}
\Biggl\{\left(\partial_{r+2-(k-1)} L\langle x_{\ast}\rangle^r(T^\prime)
+ \sum_{i=1}^{k-1} (-1)^{i} \left(\partial_{r+2-(k-1)+i} L \right)^{\Delta^i}
\langle x_{\ast}\rangle^r(T^\prime)\cdot \Psi_i^r(k)\right)\\
\times x_{\ast}^{\sigma^{k-1}\Delta^{r-k}}(T^\prime)\Biggr\} =0,
\end{multline*}
proving that $x_{\ast}$ satisfies the transversality
condition \eqref{tranversality} for all $k=1,\ldots,r$.


\subsection{Remarks and Corollaries}
\label{E-L_and_Transversality:disc}

Note that we have actually proved that, for an infinite horizon
variational problem of order $r$, one has $r$ transversality
conditions and that, for each $k=1,\ldots, r$, the $k$th transversality
condition has exactly $k$ terms. To the best of our knowledge,
even for the classical calculus of variations (\textrm{i.e.},
when $\mathbb{T}=\mathbb{R}$ or $\mathbb{T}=\mathbb{Z}$)
our explicit formulas for the transversality conditions are new.

Similarly to the special case when
$\mathbb{T} = \mathbb{R}$ (see \cite{Nitta-et-all-2009}) and
when $\mathbb{T} = \mathbb{Z}$ (see \cite{Nitta-et-all-2010}),
hypotheses 1, 2, and 3 of Theorem~\ref{main:result} are impossible to be
verified \emph{a priori} because $x_{\ast}$ is unknown. In practical terms,
such hypotheses are assumed to be true and conditions
\eqref{E-L-equation} and \eqref{tranversality} are applied heuristically
to obtain a \emph{candidate}. If such a candidate is, or not, a solution
to the problem is a different question that always require further analysis
(see Examples~\ref{ex:1} and \ref{ex:2} in Section~\ref{sec:il:ex}).

For the convenience of the reader, we present the particular cases
of Theorem~\ref{main:result} for r=1, $r=2$, and $r=3$.

\begin{corollary}{\rm \cite{MMT-2010}}
Assuming hypotheses of Theorem~\ref{main:result} for $r=1$,
if $x_\ast$ is a maximizer to problem \eqref{problem},
then $x_\ast$ satisfies the Euler--Lagrange equation
$$
(\partial_3 L)^\Delta (t, x^{\sigma}(t), x^{\Delta}(t))
= \partial_2 L (t, x^{\sigma}(t), x^{\Delta}(t) )
\quad \forall t \in [a, +\infty[ \, ,
$$
and the transversality condition
$$
\lim_{T\rightarrow+\infty} \inf_{T^\prime \geq T} \left\{
\partial_3 L(T^\prime, x^{\sigma}(T^\prime), x^{\Delta}(T^\prime))
\cdot x(T^\prime)\right\}=0.
$$
\end{corollary}

\begin{corollary}
Assuming hypotheses of Theorem~\ref{main:result} for $r=2$,
if $x_\ast$ is a maximizer to problem \eqref{problem},
then $x_\ast$ satisfies the Euler--Lagrange equation
$$
\partial_2 L\langle x \rangle^2(t)
- \left(\partial_3 L\right)^\Delta \langle x \rangle^2(t)
+ \frac{1}{a_1} \left(\partial_4 L\right)^{\Delta^2} \langle x \rangle^2(t)=0
\quad \forall t \in [a,+\infty[ \, ,
$$
and the two transversality conditions
$$
\lim_{T\rightarrow+\infty} \inf_{T^\prime \geq T}\left\{
\partial_4 L\langle x\rangle^2(T^\prime) \cdot x^{\Delta}(T^\prime)\right\}=0,
$$
$$
\lim_{T\rightarrow+\infty} \inf_{T^\prime \geq T}\left\{\left(\partial_3 L\langle x \rangle^2(T^\prime)
- \frac{1}{a_1} \left(\partial_4 L\right)^{\Delta}\langle x \rangle^2(T^\prime) \right)
\cdot x^{\sigma}(T^\prime)\right\}=0.
$$
\end{corollary}

\begin{corollary}
Assuming hypotheses of Theorem~\ref{main:result}
for $r=3$, if $x_\ast$ is a maximizer to problem
\eqref{problem}, then $x_\ast$ satisfies the Euler--Lagrange equation
$$
\partial_2 L\langle x\rangle^3(t) - \left(\partial_3 L\right)^\Delta \langle x\rangle^3(t)
+ \left(\frac{1}{a_1}\right) \left(\partial_4 L\right)^{\Delta^2} \langle x \rangle^3(t)
- \left(\frac{1}{a_1}\right)^3 \left(\partial_5 L\right)^{\Delta^3} \langle x \rangle^3(t) =0
\quad \forall t \in [a,+\infty[\, ,
$$
and the three transversality conditions
$$
\lim_{T\rightarrow+\infty} \inf_{T^\prime \geq T}\left\{
\partial_5 L\langle x \rangle^3(T^\prime) \cdot x^{\Delta^2}(T^\prime)\right\}=0,
$$
$$
\lim_{T\rightarrow+\infty} \inf_{T^\prime \geq T}\left\{\left(\partial_4 L\langle x \rangle^3(T^\prime)
- \left(\frac{1}{a_1}\right)^2 \left(\partial_5 L\right)^{\Delta}\langle x\rangle^3
(T^\prime)\right)\cdot x^{\sigma\Delta}(T^\prime)\right\}=0,
$$
$$
\lim_{T\rightarrow+\infty} \inf_{T^\prime \geq T}\left\{\left(\partial_3 L\langle x \rangle^3(T^\prime)
- \left(\frac{1}{a_1}\right) \left(\partial_4 L\right)^{\Delta} \langle x \rangle^3(T^\prime)
+ \left(\frac{1}{a_1}\right)^3 \left(\partial_5 L\right)^{\Delta^2} \langle x\rangle^3(T^\prime)
\right) x^{\sigma^2}(T^\prime)\right\}=0.
$$
\end{corollary}

Considering $\mathbb{T}=\mathbb{R}$ in Theorem~\ref{main:result},
we get the following result that improves the results
of \cite{Nitta-et-all-2009}.

\begin{corollary}
Consider the problem
\begin{equation}
\label{problem1}
\begin{gathered}
\int_{a}^{+\infty} L\left(t, x(t), x^\prime(t),
\ldots, x^{(r)}(t)\right) dt  \longrightarrow \max \\
x \in C^{r}([a,+\infty[,\mathbb{R})\\
x(a)=\alpha_0, \quad \ldots, \quad x^{(r-1)}(a)=\alpha_{r-1},
\end{gathered}
\end{equation}
where $(u_0,\ldots, u_{r})\rightarrow L(t,u_0,\ldots,u_{r})$
is a $C^1(\mathbb{R}^{r+1}, \mathbb{R})$
function for any $t \in [a,+\infty[$,
$\alpha_0, \ldots, \alpha_{r-1}$ are fixed real numbers,
and $\partial_{i+2} L\in C^{r}([a,+\infty[,\mathbb{R})$
for all $i=1, \ldots, r$ and all $x \in C^{r}([a,+\infty[,\mathbb{R})$.
Suppose that the maximizer to problem \eqref{problem1} exists and
is given by $x_{\ast}$. Let $\eta \in C^{r}([a,+\infty[,\mathbb{R}^n)$
be such that $\eta(a)=0, \ldots, \eta^{(r-1)}(a)=0$. Define
\begin{equation*}
\begin{split}
A(\varepsilon, T^\prime) &:= \int_{a}^{T^\prime}
\frac{L\left(t, x_{\ast}(t) + \varepsilon \eta(t),\ldots,
x_{\ast}^{(r)}(t) + \varepsilon \eta^{(r)}(t)\right)
- L\left(t, x_{\ast}(t),\ldots,  x_{\ast}^{(r)}(t)\right)}{\varepsilon} dt,\\
V(\varepsilon, T) &:= \inf_{T^\prime \geq T}
\int_{a}^{T^\prime} \left[L\left(t, x_{\ast}(t) + \varepsilon \eta(t),
\ldots, x_{\ast}^{(r)}(t) + \varepsilon \eta^{(r)}(t)\right)
- L(t, x_{\ast}(t),\ldots,  x_{\ast}^{(r)}(t))\right] dt, \\
V(\varepsilon) &:= \lim_{T\rightarrow+\infty} V(\varepsilon, T).
\end{split}
\end{equation*}
Suppose that
\begin{enumerate}

\item $\displaystyle \lim_{\varepsilon \rightarrow 0}
\frac{V(\varepsilon, T) }{\varepsilon}$ exists for all $T$;

\item $\displaystyle \lim_{T\rightarrow+\infty}\frac{V(\varepsilon, T)}{\varepsilon}$
exists uniformly for $\varepsilon$;

\item For every $T^\prime > a$, $T > a$,
and $\varepsilon\in \mathbb{R}\setminus\{0\}$,
there exists a sequence $\left(A(\varepsilon, T^\prime_n)\right)_{n \in \mathbb{N}}$
such that
$$
\displaystyle \lim_{n \rightarrow +\infty} A(\varepsilon, T^\prime_n)
= \displaystyle \inf_{T^\prime \geq T} A(\varepsilon, T^\prime)
$$
uniformly for $\varepsilon$.
\end{enumerate}
Then $x_\ast$ satisfies the Euler--Lagrange equation
$$
\sum_{i=0}^{r} (-1)^i \left(\partial_{i+2} L\right)^{(i)}\left(t,
x(t), x^\prime(t), \ldots, x^{(r)}(t)\right) =0
\quad \forall t \in [a,+\infty[\, ,
$$
and the $r$ transversality conditions
\begin{multline*}
\lim_{T\rightarrow+\infty} \inf_{T^\prime \geq T}
\Biggl\{\bigg( \partial_{r+2-(k-1)}
L\left(t,x(t), x^\prime(t), \ldots, x^{(r)}(t)\right)\\
+ \sum_{i=1}^{k-1} (-1)^{i}
\left(\partial_{r+2-(k-1)+i} L\right)^{(i)}(t,x(t), x^\prime(t), \ldots, x^{(r)}(t)) \bigg)
\cdot x^{(r-k)}(T^\prime)\Biggr\} =0,
\end{multline*}
$k=1, \ldots, r$.
\end{corollary}


\section{Illustrative Examples}
\label{sec:il:ex}

The following two examples illustrate the usefulness
of Theorem~\ref{main:result}.

\begin{ex}
\label{ex:1}
Let $\mathbb{T}$ be a time scale satisfying
condition $(H)$ and such that
$\sup \mathbb{T}=+\infty$
and $0 \in \mathbb{T}$. Consider the problem
\begin{equation}
\label{example:1}
\begin{gathered}
\int_{0}^{+\infty} - \left(x^{\Delta^2}(t)\right)^2 \Delta t \longrightarrow \max\\
x \in C_{rd}^4\left([0,+\infty[, \mathbb{R}\right)\\
x(0)=0 , \quad x^\Delta(0)= 1.
\end{gathered}
\end{equation}
By Theorem~\ref{main:result}, if $x_\ast$ is a maximizer to problem
\eqref{example:1}, then $x_\ast$ satisfies the Euler--Lagrange equation
\begin{equation*}
\partial_2 L\langle x \rangle^2(t) - (\partial_3 L)^\Delta \langle x \rangle^2(t)
+\frac{1}{a_1} (\partial_4 L)^{\Delta^2} \langle x \rangle^2(t) = 0
\quad \forall t \in [0,+\infty[ \, .
\end{equation*}
Since
\begin{equation*}
\partial_2L\langle x \rangle^2(t) =0, \quad \partial_3L\langle x \rangle^2(t) =0,
\quad \partial_4 L\langle x \rangle^2(t) =-2 x^{\Delta^2}(t),
\end{equation*}
then the Euler--Lagrange equation is
\begin{equation}
\label{ex:1-1}
x^{\Delta^4}(t) =0 \quad \forall t \in [0,+\infty[.
\end{equation}
Clearly,
$$
x_\ast(t)=c_1t^3 + c_2t^2 + c_3 t+ c_4,
$$
where $c_1, c_2, c_3,c_4 \in \mathbb{R}$,
is the solution of \eqref{ex:1-1}.
Using the initial conditions $x_\ast(0)=0$ and $x_\ast^\Delta(0)= 1$,
we get $c_4=0$ and $c_3=1-c_2a_0-c_1a_0^2$.
Using the two transversality conditions, we will
determine the value of $c_1$ and $c_2$. Since
$$
\lim_{T\rightarrow+\infty} \inf_{T^\prime
\geq T}\left\{\left(\partial_3 L\langle x_\ast \rangle^2(T^\prime)
- \frac{1}{a_1} \left(\partial_4 L\right)^{\Delta}\langle x_\ast
\rangle^2(T^\prime) \right)\cdot x_\ast^{\sigma}(T^\prime)\right\}=0
$$
is equivalent to
\begin{multline*}
\lim_{T\rightarrow+\infty} \inf_{T^\prime \geq T}\Bigl\{
\left(c_1 (a_1 T^\prime+ a_0)^3
+ c_2(a_1 T^\prime+ a_0)^2 + (1-c_2a_0-c_1a_0^2)(a_1 T^\prime+ a_0)\right)\\
\times c_1(1+a_1^2+a_1)(1+a_1) \Bigr\} =0,
\end{multline*}
then we conclude that $c_1=0$.
Using the transversality condition
$$
\lim_{T\rightarrow+\infty} \inf_{T^\prime \geq T}\left\{
\partial_4 L\langle x_\ast\rangle^2(T^\prime) \cdot x_\ast^{\Delta}(T^\prime)\right\}=0,
$$
that is,
$$
\lim_{T\rightarrow+\infty} \inf_{T^\prime \geq T}\left\{-2c_2(1+a_1)
\cdot \left(c_2(T^\prime+a_1T^\prime+a_0)+1-c_2a_0\right)  \right\}=0,
$$
we conclude that $c_2=0$. Hence, $x_\ast(t)=t$ is a candidate to be
a maximizer to problem \eqref{ex:1-1}. Since
\begin{multline*}
\lim_{T\rightarrow+\infty} \inf_{T^\prime \geq T}
\int_0^{T^\prime}\left[L(t, x^{\sigma^2}(t), x^{\sigma\Delta}(t), x^{\Delta^2}(t))
- L(t, x_\ast^{\sigma^2}(t), x_\ast^{\sigma\Delta}(t), x_\ast^{\Delta^2}(t))\right] \Delta t\\
= \lim_{T\rightarrow+\infty} \inf_{T^\prime
\geq T}\int_0^{T^\prime}\left[ -(x^{\Delta^2}(t))^2 \right] \Delta t \leq 0
\end{multline*}
for every admissible function $x$, then $x_\ast$
is indeed a solution to problem \eqref{ex:1-1}.
\end{ex}

In what follows, we use the standard notation of
\emph{quantum calculus} (see, \textrm{e.g.}, \cite{KacCheung}):
$$
D_q [y](t):=\displaystyle\frac{y(qt)-y(t)}{(q-1)t}
\ \quad \mbox{and} \quad D_q^{2} [y](t):=D_q [D_q[y]](t).
$$

\begin{ex}
\label{ex:2}
Fix $q>1$ and let $\mathbb{T}=q^{\mathbb{N}_0}$.
Consider the following non-autonomous problem:
\begin{equation}
\label{example:2}
\begin{gathered}
\int_{1}^{+\infty}
- t \left(1+ \left(D_q^2[x](t)\right)^2\right) \ d_q t \longrightarrow \max\\
x(1)=\alpha \, , \quad D_q[x](1)=\beta,
\end{gathered}
\end{equation}
where $\alpha$ and $\beta$ are fixed real numbers.
By Theorem~\ref{main:result}, if $x_\ast$ is a maximizer to problem
\eqref{example:2}, then $x_\ast$ must satisfy the Euler--Lagrange equation
\begin{equation*}
\partial_2 L\langle x \rangle^2(t) - D_q[\partial_3 L] \langle x \rangle^2(t)
+ \frac{1}{q} D_q^2[\partial_4 L] \langle x \rangle^2(t)  =0
\quad \forall t \in [1,+\infty[ \, .
\end{equation*}
Since
\begin{equation*}
\partial_2L\langle x \rangle^2(t) =0, \quad \partial_3L\langle x \rangle^2(t) =0,
\quad \partial_4 L\langle x \rangle^2(t) = -2 t D_q^2[x](t),
\end{equation*}
then the Euler--Lagrange equation takes the form
\begin{equation}
\label{ex:2-1}
D_q^2\left[2 t \, D_q^2[x]\right](t) = 0
\quad \forall t \in [1,+\infty[ \, .
\end{equation}
It is easy to see that $x_\ast(t)=k_1t^2+ k_2t+ k_3t \ln t +k_4$
is the general solution of equation \eqref{ex:2-1}.
Using the initial conditions we obtain
$$
k_2 = \beta- k_1(1+q) - k_3\frac{q}{q-1}\ln q
$$
and
$$
k_4=-\beta + k_1q+k_3\frac{q}{q-1}\ln q + \alpha.
$$
Using the transversality condition
$$
\lim_{T\rightarrow+\infty} \inf_{T^\prime \geq T}
\left\{\left(\partial_3 L\langle x_\ast \rangle^2(T^\prime)
- \frac{1}{q} D_q[\partial_4 L]\langle x_\ast \rangle^2
(T^\prime) \right)\cdot x_\ast(qT^\prime)\right\}=0
$$
we get $k_3=0$. The transversality condition
$$
\lim_{T\rightarrow+\infty} \inf_{T^\prime \geq T}
\left\{\partial_4 L\langle x_\ast\rangle^2(T^\prime) \cdot D_q[x_\ast](T^\prime)\right\}=0
$$
implies that $k_1=0$. Hence, $x_\ast(t)=\beta t-\beta + \alpha$
is a candidate to be a maximizer. Using the definition of weak maximality,
we conclude that $x_\ast(t)=\beta t-\beta + \alpha$
is indeed a maximizer to problem \eqref{example:2}.
\end{ex}


\section{Concluding Remarks}
\label{sec:conc}

We have established Euler--Lagrange and transversality optimality
conditions for higher-order infinite horizon variational
problems on a time scale. The results were obtained for
weakly optimal solutions in the Brock sense.
The main result is Theorem~\ref{main:result}, which
generalizes the recent results of \cite{MMT-2010} and \cite{naty:irlanda}.
Moreover, the new necessary optimality conditions improve
the classical results both in the continuous and discrete settings:
if one chooses the time scale to be the set of real numbers, then
Theorem~\ref{main:result} improves the continuous results of \cite{Nitta-et-all-2009};
if one chooses the time scale to be the set of integers, then Theorem~\ref{main:result}
improves the discrete-time results of \cite{Nitta-et-all-2010}.



\end{document}